\scrollmode

\documentclass{amsart}
\usepackage{amssymb}
\usepackage[all]{xy}
\usepackage{url}

\theoremstyle{plain}
\newtheorem{prop}{Proposition}
\newtheorem{thm}[prop]{Theorem}
\newtheorem{cor}[prop]{Corollary}
\newtheorem{lem}[prop]{Lemma}

\newtheorem*{thmA}{Theorem A}
\newtheorem*{corB}{Corollary B}
\newtheorem*{thmC}{Theorem C}

\theoremstyle{definition}

\theoremstyle{remark}

\numberwithin{prop}{section}
\numberwithin{ques}{section}
\numberwithin{equation}{section}

\DeclareMathOperator{\cl}{cl}

\newcommand{\ca}[1]{\mathcal{#1}}

\newcommand{\Z}{\mathbb{Z}}

\newcommand{\triv}{\{1\}}                                   

\newcommand{\caB}{\ca{B}}
\newcommand{\caC}{\ca{C}}

\newcommand{\tU}{\widetilde{U}}
\newcommand{\tV}{\widetilde{V}}

\newcommand{\argu}{\hbox to 7truept{\hrulefill}}
\newcommand{\G}{\ca{G}}

\newcommand{\rdd}{\mathrm{red}}
\newcommand{\bE}{\mathbf{E}}

\begin{document}
\title[Virtually free pro-p products]{Virtually free pro-p products}
\author{Th. Weigel and P. A. Zalesski\u i}
\date{\today}
\address{Th. Weigel\\
Universit\`a di Milano-Bicocca\\
U5-3067, Via R.Cozzi, 55\\
20125 Milano, Italy}
\email{thomas.weigel@unimib.it}

\address{P. A. Zalesski\u i\\
Department of Mathematics\\
University of Brasilia\\
70.910 Brasilia DF\\
Brazil}
\email{pz@mat.unb.br}

\begin{abstract}
It is shown that a finitely generated pro-$p$ group $G$ which is a
virtually free pro-$p$  product splits either as a free pro-$p$
product with amalgamation or as a pro-$p$ HNN-extension over a
finite $p$-group. More precisely, $G$ is the pro-$p$ fundamental
group of a finite graph of finitely generated pro-$p$ groups with
finite edge groups. This generalizes previous results of
W.~Herfort and the second author (cf. \cite{wolfpav:virt}).
\end{abstract}


\maketitle

\section{Introduction}
\label{s:intro}
In 1965, J-P.~Serre showed that a torsion free virtually free pro-$p$ group must be free
(cf. \cite{ser:coho}). This motivated him to ask the question whether the same statement holds
also in the discrete context. His question was answered positively some years later.
In several papers (cf. \cite{stall:bull}, \cite{stall:ends}, \cite{swan:coh1}),
J.R.~Stallings and R.G.~ Swan showed that free groups are precisely the groups of cohomological dimension $1$, and at the same time
J-P.~Serre himself showed that in a torsion free group $G$ the cohomological dimension of a
subgroup of finite index coincides with the cohomological dimension of $G$ (cf. \cite{ser:cohdis}).

One of the major tools for obtaining this type of result - the theory of
ends - provided deep results also in the presence of torsion. The
first result to be mentioned is `Stallings' decomposition theorem'
(cf. \cite{stall:deco2}). It generalizes the previously mentioned
result to virtual free products.

\begin{thm}[J.R.~Stallings]
\label{thm:stall}
Let $G$ be a finitely generated group containing a subgroup of finite index
which is a non-trivial free product. Then $G$ splits either as a free product with amalgamation or as an HNN-extension over a finite group.
\end{thm}

The purpose of this paper is to prove a pro-$p$ analogue of Theorem~\ref{thm:stall}.

\begin{thmA}
Let $G$ be a finitely generated pro-$p$ group containing an open subgroup
which is a non-trivial free pro-$p$ product. Then $G$ splits either as a free pro-$p$ product with amalgamation or as a pro-$p$ HNN-extension over a finite $p$ group.
\end{thmA}

In the torsion free case Theorem A yields a splitting  of $G$
into a non-trivial free pro-$p$ product.

\begin{corB}
Let $G$ be a finitely generated torsion free
pro-$p$ group which is a virtual free pro-$p$ product. Then $G$ is
a non-trivial free pro-$p$ product.
\end{corB}

In contrast to the proof of Theorem~\ref{thm:stall} which uses the theory of ends,
the proof of Theorem~A is accomplished
by using purely combinatorial methods in pro-$p$ group theory, and the description of finitely generated
virtually free pro-$p$ groups obtained by W.~Herfort and the second author (cf. \cite{wolfpav:virt}).
In fact, the techniques of pro-$p$ groups acting on pro-$p$ trees are used in order to obtain the following
more conceptual version of Theorem~A
(cf. Thm.~\ref{thm:virtfreeprod}).

\begin{thmC}
 Let $G$ be a finitely generated pro-$p$
group containing an open subgroup $H$ which has a non-trivial
decomposition as free product, i.e., there exists non-trivial
closed subgroups $A,B\subsetneq H$ such that $H=A\amalg B$. Then
$G$ is isomorphic to the pro-$p$ fundamental group of a finite
graph of pro-$p$ groups with finite edge stabilizers.
\end{thmC}

Two achievements had caused dramatic advances in
the combinatorial theory of groups; Bass-Serre theory of groups acting on trees
and `Stallings' decomposition theorem' of groups with infinitely many ends.
The results of this paper contribute to the theory of pro-$p$ groups acting on
pro-$p$ trees. Nevertheless, the absense of a `Stallings' decomposition theorem'
in the pro-$p$ context is still overshadowing the combinatorial theory of pro-$p$ groups.

\section{Preliminaries}
\label{s:prelim}
We will use the notion of graph as introduced by
J-P.~Serre in \cite[\S 2.1]{ser:trees}.

\subsection{Finite graphs of pro-$p$ groups}
\label{ss:grofprop}
Let $\Gamma$ be a finite connected graph.
A {\it graph of groups} $\G$ based on $\Gamma$
is called a {\it finite graph of pro-$p$ groups},
if all vertex groups $\G(v)$, $v\in V(\Gamma)$, and all
edge groups $\G(e)$, $e\in E(\Gamma)$, are pro-$p$ groups,
and if all the group monomorphisms $\alpha_e\colon\G(e)\to\G(t(e))$
are continuous. So, if $(\G,\Gamma)$ is an (abstract)
graph of groups such that
all vertex and edge groups are finitely generated pro-$p$ groups, then
by a theorem of J-P.~Serre (cf. \cite[\S 4.8]{ribzal:prof}),
$(\G,\Gamma)$ is a finite graph of pro-$p$ groups.

A finite graph of pro-$p$ groups $(\G,\Gamma)$ is said to be
{\it reduced}, if for every geometric edge $\{e,\bar{e}\}$ which is not  a loop
neither $\alpha_e\colon\G(e)\to\G(t(e))$
nor $\alpha_{\bar{e}}\colon \G(e)\to\G(o(e))$ is an isomorphism.
Any finite graph of pro-$p$ groups can be transformed in a reduced finite graph of pro-$p$ groups
by the following procedure: If $\{e,\bar{e}\}$ is a geometric edge which is not a loop, we can
remove $\{e,\bar{e}\}$ from the edge set of $\Gamma$, and identify $o(e)$ and $t(e)$
in a new vertex $y$. Let $\Gamma^\prime$ be the finite graph given by
$V(\Gamma^\prime)=\{y\}\sqcup V(\Gamma)\setminus\{o(e),t(e)\}$ and
$E(\Gamma^\prime)=E(\Gamma)\setminus\{e,\bar{e}\}$, and let
$\G^\prime$ denote the finite graph of pro-$p$ groups based on $\Gamma^\prime$ given by
$\G^\prime(y)=\G(o(e))$
if $\alpha_e$ is an isomorphism, and $\G^\prime(y)=\G(t(e))$ if $\alpha_e$ is not an isomorphism.
This procedure can be continued until $\alpha_e$ is not surjective for all edges not defining loops.
The resulting finite graph of pro-$p$ groups $(\G_{\rdd},\Gamma_{\rdd})$ is reduced.


\subsection{The fundamental pro-$p$ group of a finite graph of finitely generated pro-$p$ groups}
\label{ss:proptree} Let $(\G,\Gamma)$ be a finite graph of
finitely generated pro-$p$ groups. We define the {\it fundamental
pro-$p$ group} $G=\Pi_1(\G,\Gamma, v_0)$, $v_0\in V(\Gamma)$, of
$(\G,\Gamma)$ to be the pro-$p$ completion of the usual
fundamental group $\pi_1(\G,\Gamma, v_0)$ 
(cf. \cite[\S5.1]{ser:trees}). In general, $\pi_1(\G,\Gamma, v_0)$ does not
have to be residually $p$, but this will be the case in all of our
considerations. In particular, edge and vertex groups will be
subgroups of $\Pi_1(\G,\Gamma, v_0)$. Since $\G(e)$, $\G(v)$ are
finitely generated, by a theorem of J-P.~Serre (cf. \cite[\S 4.8]{ribzal:prof}),
our definition is equivalent to the original
definition of the fundamental group of a graph of groups in the
category of pro-$p$ groups (cf. \cite{melpav:trees}). Note that
the previously mentioned reduction process does not change the
fundamental pro-$p$ group, i.e., one has a canonical isomorphism
$\Pi_1(\G,\Gamma, v_0)\simeq \Pi_1(\G_{\rdd},\Gamma_{\rdd},w_0)$.
So, if the pro-$p$ group $G$ is the fundamental group of a finite
graph of pro-$p$ groups, we may assume that the finite graph of
pro-$p$ groups is reduced.


\subsection{The fundamental pro-$p$ group of a finite graph of finite $p$-groups}
\label{ss:finsgrp}
Let $(\G,\Gamma)$ be a finite graph of finite $p$-groups.
By  \cite[Thm.~ 3.10]{melpav:trees}, every finite subgroup of $G=\Pi_1(\G,\Gamma, v_0)$
is conjugate to a subgroup
of a vertex group of $(\G,\Gamma)$.  Hence $G$ has only finitely many finite subgroups up to conjugation.
In particular, every maximal finite subgroup of $G$ is $G$-conjugate to a vertex
group of $(\G,\Gamma)$, and the converse is true if $(\G,\Gamma)$ is a reduced finite graph of
finite $p$-groups.


\section{Virtually free pro-$p$ products}
\label{s:vfreeprop}


\subsection{Virtually free pro-$p$ groups}
\label{ss:virfreeprop} 
A pro-$p$ group $G$ will be called to be a 
{\it free pro-$p$ product} if there exist non-trivial closed subgroups
$A$ and $B$ such that $G=A\amalg B$. Otherwise we shall say that
$G$ is $\amalg$-indecomposable.
The following properties are well known.

\begin{prop}
\label{prop:freeprod} 
Let $H=\coprod_{i\in I} H_i\coprod F$
be a finitely generated pro-$p$ group with a $\amalg$-decompositon,
where $H_i$ are non-trivial $\amalg$-indecomposable pro-$p$-groups,  
and $F$ is a free pro-$p$ group. Then
\begin{itemize}
\item[(a)] $I$ is finite, and $H_i$, $i\in I$, and $F$ are finitely generated.
\item[(b)] Any finitely generated $\amalg$-indecomposable subgroup $A$
of $H$ is conjugate to a subgroup of $H_i$ for some $i\in I$.
Moreover, if $H=A\amalg B$ for some closed subgroup $B$ of $H$, then $A$ is conjugate to some $H_i$, $i\in I$. 
\item[(c)] $H_i\cap H_j^h=1$ if either $i\neq j$ or $h\not\in H_i$.
\item[(d)] For $K\subseteq H_i$, $K\not=\triv$, one has
$N_H(K)\subseteq H_i$. In particular, if $H_i$ is finite, so is
$N_H(K)$.
\end{itemize}
\end{prop}

\begin{proof}  (a) is obvious. The first statement of (b) follows from the pro-$p$ version of the
Kurosh subgroup theorem \cite[Thm.~4.4]{herrib:subgr} and the
second statement from  \cite[Thm.~4.3]{mel:quot}.
For (c) see Theorems 4.2 (a) and 4.3 (a) in \cite{ribzal:hori}.
In order to prove (d) take $h\in N_G(K)$. Then $K\subseteq H_i\cap
H_i^h$, and, by (c), one has $h\in H_i$.
\end{proof}

From Proposition~\ref{prop:freeprod} one concludes the following
properties for virtual free pro-$p$ products.

\begin{prop}
\label{prop:euler}
Let $(\G,\Gamma)$ be a reduced finite graph of finite $p$-groups, and
suppose that $G=\Pi_1(\G,\Gamma, v_0)$  contains an open, normal subgroup
$H=F\amalg H_1\amalg\cdots \amalg H_s$, $H_i\not=\triv$, for some
free pro-$p$ subgroup $F$ of rank $r$, $0\leq r<\infty$, such that $r+s\geq 2$.
Then one has the following.
\begin{itemize}
\item[(a)] For any edge $e$ of $\Gamma$ one has $\G(e)\cap H=\triv$; in particular, $|\G(e)|\leq |G:H|$.
\item[(b)] $|\bE(\Gamma)|\leq 2(r+s)-1$ and $|V(\Gamma)|\leq 2(r+s)$, where $V(\Gamma)$ is the
set of vertices of $\Gamma$, and $\bE(\Gamma)$ is the set of geometric edges of $\Gamma$.
\end{itemize}
\end{prop}

\begin{proof}
Let $X=\pi_1(\G,\Gamma, v_0)$ be the abstract fundamental group
of the graph of groups and $Y=X\cap H$. Hence $G$ and $H$ are the
pro-$p$ completions of $X$ and $Y$, respectively. Moreover,
$|X:Y|=|G:H|$.

\noindent
(a) Suppose that $\G(e)\cap H\not=\triv$.
Since $H$ is normal in $G$,
$N_{G}(\G(e))$ normalizes $\G(e)\cap H$.
We claim that $N_G(\G(e))$ is infinite. One has to
distinguish two cases:
Case 1:  $\{e,\bar{e}\}$ is not a loop. In this case $N_G(\G(e))$ contains the infinite group
$\langle N_{\G(v)}(\G(e)),N_{\G(w)}(\G(e))\rangle$,
where $v=o(e)$, $w=t(e)$.
Case 2: $\{e,\bar e\}$ is a loop.  Let $v=t(e)=o(e)$, and let
$z_e\in G$ be the stable letter associated with $e$. If $\G(e)=\G(v)$, then
$N_G(\G(e))$ contains the infinite group $\langle z_e\rangle$.
Otherwise $N_G(\G(e))$ contains the infinite group
$\langle N_{\G(v)}(\G(e)),z_e N_{\G(v)}(\G(e))z_e^{-1}\rangle$.

Since $|G:H|<\infty$, the fact that $N_G(\G(e))$ is infinite
implies that $N_{H}(\G(e)\cap H)=N_{G}(\G(e)\cap H)\cap H$
is infinite as well contradicting Proposition~\ref{prop:freeprod}(d).
Hence one has $\G(e)\cap H=\triv$ as required.

\noindent
(b) It suffices to show the first inequality.
By \cite[\S 2.6, Ex.~3]{ser:trees}, one has
\begin{equation}
\label{eq:eulergraph}
\begin{aligned}
-\chi_X&=\sum_{e\in \bE(\Gamma)}\frac{1}{|\G(e)|}-
\sum_{v\in V(\Gamma)}\frac{1}{|\G(v)|}\\
&=-\frac{1}{|X:Y|}\cdot\chi_Y\\
&=\frac{1}{|X:Y|}\cdot\Big(r+s-1-\sum_{1\leq i\leq s}\frac{1}{|H_i|}\Big),
\end{aligned}
\end{equation}
where $\chi_X$ denotes the Euler characteristic of the finitely generated virtually free group
$X$. Thus one obtains
\begin{align}
r+s-1&\geq |X:Y|\Big(\sum_{e\in \bE(\Gamma)}\frac{1}{|\G(e)|}-
\sum_{v\in V(\Gamma)}\frac{1}{|\G(v)|}\Big).\label{eq:eulergraph3}
\intertext{As $(\G,\Gamma)$ is reduced,
for every edge $e$ in a maximal subtree $T$ of $\Gamma$ the edge group $\G(e)$ is isomorphic
to a proper subgroup of $\G(t(e))$.
Hence, $|\G(t(e))|\geq 2|\G(e)|$. 
Let $E^+(T)$ be an orientation of $T$ such that every vertex of $\Gamma$ except $v_0\in V(\Gamma)$
is the terminus of precisely one edge of $T$, and let $f\in E(T)$ be an edge satisfying $t(f)=v_0$.
Taking into account that $|\bE(T)|=|V(\Gamma)|-1$,
one concludes from (a) that}
r+s-1&\geq \frac{1}{2}\cdot\sum_{e\in \bE(\Gamma)\setminus\{f,\bar{f}\}}\frac{|X:Y|}{|\G(e)|}\geq
\frac{1}{2}\cdot (|\bE(\Gamma)|-1).\label{eq:eulergraph4}
\end{align}
This yields the claim.
\end{proof}

From Proposition~\ref{prop:euler} one concludes the following straightforward fact.

\begin{cor}
\label{cor:bound}
Let $(\G,\Gamma)$ be a reduced finite graph of finite $p$-groups, and
suppose that $G=\Pi_1(\G,\Gamma, v_0)$  contains a free open subgroup
$H$ of rank $r\geq 2$.
Then there exist finitely many reduced finite graphs of finite $p$-groups
$(\G^\prime,\Gamma^\prime)$ up to isomorphism
such that $G\simeq \Pi_1(\G^\prime,\Gamma^\prime, w_0)$.
\end{cor}

Let $G=\Pi_1(\G,\Gamma,v_0)$ be the pro-$p$ fundamental group of a finite graph of finite
$p$-groups, and let $U$ be an open and normal subgroup of $G$. Then, by construction,
$\tU=\cl(\langle\,U\cap {}\G(v)^g\mid g\in G,\ v\in
V(\Gamma)\,\rangle)$ is a closed normal subgroup of $G$. By
\cite[Prop.~1.10]{ribzal:iber}, one has a natural decomposition of
$G/\tU$ as the pro-$p$ fundamental group $G/\tU=\Pi_1(\G_{U},\Gamma, v_0)$
of a finite graph of finite $p$-groups $(\G_{U},\Gamma)$, where the vertex and edge groups satisfy
$\G_{U}(x)=\G(x)\tU/\tU$, $x\in V(\Gamma)\sqcup E(\Gamma)$. Thus we have a
morphism $\eta\colon (\G,\Gamma)\longrightarrow (\G_{U},\Gamma)$
of graphs of groups such that the induced homomorphism on the
pro-$p$ fundamental groups coincides with the canonical projection
$\varphi_{U}\colon G\longrightarrow G/\tU$.

\begin{lem} 
\label{lem:connectmap} 
Let $G=\Pi_1(\G,\Gamma,v_0)$ be the pro-$p$ fundamental group of a finite graph of finite
$p$-groups, and let $H$ be an open normal subgroup of 
$G$ that decomposes as a free pro-$p$ product 
$H=\coprod_{1\leq i\leq s} H_i\coprod F$ of
finite $p$-groups $H_i$ and a free pro-$p$ group $F$.
Let $U\subseteq H$ be an open normal subgroup of $G$ such that 
$U\cap H_i\neq H_i$ for every $i\in \{1,\ldots,s\}$. If
$(\G,\Gamma)$ is reduced, then $(\G_{U},\Gamma)$ is reduced.
\end{lem}

\begin{proof} Suppose on the contrary that
there exists an edge $e$ in $\Gamma$ which is not a loop such that
for $v=t(e)$ one has $\G(v)\tU=\G(e)\tU\subseteq G/\tU$. Then, by
the second isomorphism theorem,
\begin{equation}
\label{eq:claim} \G(v)=\G(e)(\tU\cap \G(v)).
\end{equation}
As $(\G,\Gamma)$ is reduced, and thus $\G(e)\neq\G(v)$, one has
$\tU\cap \G(v)\not=\triv$.  From
Proposition~\ref{prop:freeprod}(a) one deduces that $\tU\cap
\G(v)$ is contained in some $H_i^g$ for $1\leq i\leq s$ and $g\in G$.
If $ N_{G}(\tU\cap \G(v))$ would be infinite, so would be  $
N_{H}(\tU\cap \G(v))$ contradicting
Proposition~\ref{prop:freeprod}(d). Hence $N_{G}(\tU\cap \G(v))$ is finite and equal to
$\G(v)$. In particular, for $y\in\G(v)$ one concludes that
$H_i^{gy}\cap H_i^g\not=\triv$. Hence, by
Proposition~\ref{prop:freeprod}(c), $H_i^{gy}=H_i^g$ and thus
$\G(v)\subseteq N_{G}(H_i^g)$. The maximality of $\G(v)$ and the
finiteness of $N_{G}(H_i^g)$ (cf. Prop.~\ref{prop:freeprod}(d))
imply that $\G(v)= N_{G}(H_i^g)$. By construction, $\G(e)H_i^g$
is a finite subgroup of $G$ containing $\G(v)$ (cf.
\eqref{eq:claim}). As $\G(v)$ is a maximal finite subgroup of $G$, this implies that
\begin{equation}
\label{eq:claim1}
\G(e)(\tU\cap\G(v))=\G(v)=\G(e)\,H_i^g.
\end{equation}
Since $\tU\cap\G(v)\subseteq H_i^g$, and as $\G(e)\cap
H_i^g=\triv$ (cf. Prop.~\ref{prop:euler}(a)), one concludes that
$\tU\cap\G(v)=H_i^g$.
Hence $H_i\subseteq\tU\subseteq U$
contradicting the hypothesis.
\end{proof}

The proof of the structure theorem for virtual free pro-$p$ products
(cf. Thm.~\ref{thm:virtfreeprod}) in the subsequent subsection
is based on the following result due to W.~Herfort and the second author.

\begin{thm}
\label{thm:virt}
\textup{(cf. \cite[Thm.~1.1]{wolfpav:virt})}
Let $G$ be a finitely generated pro-$p$ group with a free open subgroup $F$.
Then $G$ is the pro-$p$ fundamental group of a finite graph of finite $p$-groups
whose orders are bounded by $|G:F|$.
\end{thm}


\subsection{Virtual free pro-$p$ products}
\label{ss:virtppp}
The following theorem gives a description of the structure of
virtual free pro-$p$ products.

\begin{thm}
\label{thm:virtfreeprod}
Let $G$ be a finitely generated pro-$p$ group containing an open subgroup
$H$ which has a non-trivial decomposition as free product,
i.e., there exists non-trivial closed subgroups $A,B\subsetneq H$
such that $H=A\amalg B$.
Then $G$ is isomorphic to the pro-$p$ fundamental group of
a finite graph of pro-$p$ groups with finite edge stabilizers.
\end{thm}

\begin{proof}
By replacing $H$ by the core of $H$ in $G$ and applying the
Kurosh subgroup theorem for open subgroups (cf. \cite[Thm.~9.1.9]{ribzal:prof}),
we may assume that $H$ is normal in $G$. Refining the free decomposition if necessary
and collecting free factors isomorphic to $\Z_p$ we obtain a free decomposition 
\begin{equation}
\label{eq:Hdeco}
H=F\amalg H_1\amalg\cdots\amalg H_s,
\end{equation}
where $F$ is a free subgroup of rank $t$, and the $H_i$ are
$\amalg$-indecomposable finitely generated subgroups which are not
isomorphic to $\Z_p$ (cf. Prop.~\ref{prop:freeprod}(a)). By hypothesis,  $s+t\geq 2$. By
construction, one has for all $g\in G$ and for all
$i\in\{1,\ldots,s\}$ that $H_i^g$ is a free factor of $H$. Since
$H_i$ is indecomposable, we deduce from Proposition~\ref{prop:freeprod}(b)  
that the indecomposable non-free subgroup
$H_i^g$ of $H$ equals $H_j^h$ for some $j\in\{1,\ldots,s\}$. Thus
$\{H_i^g\mid g\in G,\ 1\leq i\leq s\,\}=\{H_i^h\mid h\in H,\ 1\leq
i\leq s\,\}$.

\noindent
{\bf Step 1:} Let $\caB$ be a basis of neighbourhoods of $1_G\in G$ consisting of open normal
subgroups $U$ of $G$ which are contained in $H$ with $H_i\not\subseteq U$ for every $i=1,\ldots s$. 
For $U\in\caB$ put
\begin{equation}
\label{eq:tU}
\tU=\cl(\langle\,U\cap {}H_i^g\mid g\in G,\ 1\leq i\leq s\,\rangle)=
\cl(\langle\,U\cap {}H_i^h\mid h\in H,\ 1\leq i\leq s\,\rangle).
\end{equation}
Then $\tU$ is a closed normal subgroup of $H$, and
\begin{equation}
\label{eq:defHtU}
H/\tU= F\amalg H_1\tU/\tU\amalg\cdots\amalg H_s\tU/\tU
\end{equation}
(cf. \cite[Prop.~1.18]{mel:quot}). The group $G/\tU$ contains the
open normal subgroup $H/\tU$ which is a finitely generated,
virtually free pro-$p$ group (since $U/\tU$ is free pro-$p$ by
Theorem 2.6 in \cite{melpav:trees}), and thus $G/\tU$ is a
finitely generated, virtually free pro-$p$ group.

\noindent
{\bf Step 2:} By Theorem \ref{thm:virt}, $G/\tU$  is isomorphic to the pro-$p$ fundamental group
$\Pi_1(\G_U,\Gamma_U, v_U)$ of a finite graph
of finite $p$-groups. Using the procedure described in subsection \ref{ss:proptree}
we may assume that $(\G_U,\Gamma_U)$ is
reduced. Hence from now on we may assume that for every $U\in \caB$ the vertex groups of
$G/\tU=\Pi_1(\G_U,\Gamma_U, v_U)$ are representatives of the
$G/\tU$-conjugacy classes of maximal finite subgroups.
Note that by Proposition~\ref{prop:euler}(a), one has $\G_U(e)\cap H/\tilde U=1$.

\noindent {\bf Step 3:} As explained before Lemma~\ref{lem:connectmap},
for $V\subseteq U$ both open and normal in
$G$ the decomposition $G/\tV=\Pi_1(\G_V,\Gamma_V, v_V)$ gives rise
to a natural decomposition of $G/\tU$ as the fundamental group
$G/\tU=\Pi_1(\G_{V,U},\Gamma_V, v_V)$ of a graph of groups
$(\G_{V,U},\Gamma_V)$. Moreover, by Lemma~\ref{lem:connectmap},
if $(\G_V,\Gamma_V)$ is reduced, then $(\G_{V,U},\Gamma_V)$ is reduced. Thus in this case one has 
a morphism
$\eta\colon (\G_{V},\Gamma_V)\longrightarrow (\G_{V,U},\Gamma_V)$
of reduced graph of groups such that the induced homomorphism on
the pro-$p$ fundamental groups coincides with the canonical
projection $\varphi_{UV}\colon G/\tV\longrightarrow G/\tU$.

\noindent {\bf Step 4:}  By Proposition~\ref{prop:euler}, the
number $|V(\Gamma_U)|+|\bE(\Gamma_U)|$ is bounded by $4(r+s)-1$.
Therefore, by passing to a cofinal system $\caC$ of $\caB$ if
necessary, we may assume that $\Gamma_U=\Gamma$ for each $U\in
\caC$. Then, by Corollary ~\ref{cor:bound},  the number of
isomorphism classes of finite reduced graphs of finite $p$-groups
$(\G^\prime_U,\Gamma)$ which are based on $\Gamma$ and satisfy
$G/\tU\simeq\Pi_1(\G^\prime,\Gamma,v_0)$ is finite. Suppose that
$\Omega_U$ is a set containing a copy of every such isomorphism
class. For $V\in\caC$, $V\subseteq U$, one has a map
$\omega_{V,U}\colon \Omega_V\to\Omega_U$ (cf. Step~3). Hence
$\Omega=\varprojlim_{U\in\caC}\Omega_U$ is non-empty. Let
$(\G^\prime_U,\Gamma)_{U\in\caC}\in\Omega$. Then
$(\G^\prime,\Gamma)$ given by $\G^\prime(x)=\varprojlim
\G^\prime_U(x)$ if $x$ is either a vertex or an edge of $\Gamma$,
is a reduced finite graph of finitely generated pro-$p$ groups
satisfying $G\simeq\Pi_1(\G^\prime,\Gamma,v_0)$. By
Proposition~\ref{prop:euler}(a), $\G^\prime(e)$ is finite for
every edge $e$ of $\Gamma$. This yields the claim.
\end{proof}





\begin{thebibliography}{10}

\bibitem{herrib:subgr}
W.~Herfort and L.~Ribes, \emph{Subgroups of free pro-{$p$}-products}, Math.
  Proc. Cambridge Philos. Soc. \textbf{101} (1987), no.~2, 197--206. \MR{870590
  (87m:20083)}

\bibitem{wolfpav:virt}
W.~Herfort and P.~A. Zalesskii, \emph{Virtually free pro-p groups}, Publ. Math.
  Inst. Hautes {\'E}tudes Sci. \textbf{118} (2013), 193--211. \MR{3150249}

\bibitem{mel:quot}
O.~V. Mel{'}nikov, \emph{Subgroups and the homology of free products of
  profinite groups}, Izv. Akad. Nauk SSSR Ser. Mat. \textbf{53} (1989), no.~1,
  97--120. \MR{992980 (91b:20033)}

\bibitem{ribzal:hori}
L.~Ribes and P.~A. Zalesskii, \emph{Pro-{$p$} trees and applications}, New
  horizons in pro-{$p$} groups, Progr. Math., vol. 184, Birkh{\"a}user Boston,
  Boston, MA, 2000, pp.~75--119. \MR{1765118 (2001f:20057)}

\bibitem{ribzal:prof}
\bysame, \emph{Profinite groups}, Ergebnisse der Mathematik und ihrer
  Grenzgebiete. 3. Folge. A Series of Modern Surveys in Mathematics [Results in
  Mathematics and Related Areas. 3rd Series. A Series of Modern Surveys in
  Mathematics], vol.~40, Springer-Verlag, Berlin, 2000. \MR{1775104
  (2001k:20060)}

\bibitem{ribzal:iber}
L.~Ribes and P.~A. Zalesski{\u\i}, \emph{Normalizers in groups and their
  profinite completion}, Rev. Mat. Iberoam. \textbf{30} (2014), no.~1,
  167--192.

\bibitem{ser:coho}
J-P. Serre, \emph{Sur la dimension cohomologique des groupes profinis},
  Topology \textbf{3} (1965), 413--420. \MR{0180619 (31 \#4853)}

\bibitem{ser:cohdis}
\bysame, \emph{Cohomologie des groupes discrets}, Prospects in mathematics
  ({P}roc. {S}ympos., {P}rinceton {U}niv., {P}rinceton, {N}.{J}., 1970),
  Princeton Univ. Press, Princeton, N.J., 1971, pp.~77--169. Ann. of Math.
  Studies, No. 70. \MR{0385006 (52 \#5876)}

\bibitem{ser:trees}
\bysame, \emph{Trees}, Springer-Verlag, Berlin, 1980, Translated from the
  French by John Stillwell. \MR{607504 (82c:20083)}

\bibitem{stall:bull}
J.~R. Stallings, \emph{Groups of dimension 1 are locally free}, Bull. Amer.
  Math. Soc. \textbf{74} (1968), 361--364. \MR{0223439 (36 \#6487)}

\bibitem{stall:ends}
\bysame, \emph{On torsion-free groups with infinitely many ends}, Ann. of Math.
  (2) \textbf{88} (1968), 312--334.

\bibitem{stall:deco2}
\bysame, \emph{Group theory and three-dimensional manifolds}, Yale University
  Press, New Haven, Conn.-London, 1971, A James K. Whittemore Lecture in
  Mathematics given at Yale University, 1969, Yale Mathematical Monographs, 4.
  \MR{0415622 (54 \#3705)}

\bibitem{swan:coh1}
R.~G. Swan, \emph{Groups of cohomological dimension one}, J. Algebra
  \textbf{12} (1969), 585--610. \MR{0240177 (39 \#1531)}

\bibitem{melpav:trees}
P.~A. Zalesski{\u\i} and O.~V. Mel{'}nikov, \emph{Subgroups of profinite groups
  acting on trees}, Mat. Sb. (N.S.) \textbf{135(177)} (1988), no.~4, 419--439,
  559. \MR{942131 (90f:20041)}

\end{thebibliography}
\end{document}